\documentclass[reqno]{amsart}

\usepackage{amssymb}
\usepackage{amsmath}
\usepackage{color}
\usepackage{graphicx,epsfig}
\usepackage{mathrsfs}
\usepackage{amstext}
\usepackage{latexsym}

\usepackage[labelformat=empty]{caption}
\usepackage{enumerate}

\usepackage{enumitem}
\usepackage[pdftex]{hyperref}
\usepackage[pdftex,open]{bookmark}
\hypersetup{breaklinks=true,
            colorlinks=false,
            citecolor=green,
            urlcolor=blue,
            linkcolor=blue,
            pdfpagemode=UseOutlines,
            pdftitle={An New Optimal Inequality},
            pdfauthor={Omran Kouba}}

\theoremstyle{plain}
\newtheorem{theorem}{Theorem}[section]

\newtheorem{lemma}[theorem]{Lemma}

\newtheorem{proposition}[theorem]{Proposition}
\theoremstyle{remark}

\newcommand{\reel}{\mathbb{R}}

\newcommand{\ent}{\mathbb{Z}}

\newcommand{\abs}[1]{\left\vert #1\right\vert }

\newcommand{\bg}{\medskip\goodbreak}

\newcommand{\vers}{{\,\longrightarrow\,}}
\newcommand{\egdef}{\buildrel{\scriptscriptstyle{\rm def}}\over{=}}

\newcommand{\ba}{\operatorname{\textbf{a}}}
\newcommand{\bb}{\operatorname{\textbf{b}}}

\begin{document}
\title[A Reverse Hilbert-like Optimal Inequality]
{A Reverse Hilbert-like Optimal Inequality}
\author[Omran Kouba]{Omran Kouba$^\dag$}
\address{Department of Mathematics \\
Higher Institute for Applied Sciences and Technology\\
P.O. Box 31983, Damascus, Syria.}
\email{\href{mailto:omran_kouba@hiast.edu.sy}{omran\_kouba@hiast.edu.sy}}
\keywords{inequalities, Fourier series, Fourier transform.}
\subjclass[2010]{42A16, 42A38, 42B20.}
\thanks{$^\dag$ Department of Mathematics, Higher Institute for Applied Sciences and Technology.}

\begin{abstract}
We prove an  inequality on positive real numbers, that looks like a reverse to the well-known Hilbert inequality, and we use some unusual techniques from Fourier analysis to prove that this inequality is optimal.\par
\end{abstract}

\maketitle
\section{\sc Introduction and Notation}\label{sec1}

This research was initiated by a proposed problem to the American Mathematical Monthly \cite{dal}, where it was asked to prove that
\[
\left(\sum_{k=1}^n\frac{a_j }{b_j}\right)^2-2
\left(\sum_{j,k=1}^n\frac{a_ja_k}{(b_j+b_j)^2}\right)^2\leq
2\left( \sum_{j,k=1}^n\frac{a_ja_k}{(b_j+b_k)} 
 \sum_{l,m=1}^n\frac{a_la_m}{(b_l+b_m)^3}\right)^{1/2}
\]
for positive real numbers $a_1,\ldots,a_n$ and $b_1,\ldots,b_n$. Our aim is not to prove or to discuss this inequality, but to notice that its form suggests the possibility of a typographic error in the denominator of the second term on  left,  should it be $(b_j+b_k)^2$ instead of 
$(b_j+b_j)^2$ ?  
In this note we show that the rectified version of this inequality does not hold, but rather another one with a larger
constant on the right side, and we will show this constant is the best possible. So, let us fix some notation and describe this work.
\bg  
 For a positive integer $n$ and two vectors
$\ba=(a_1,\ldots,a_n)$ and $\bb=(b_1,\ldots,b_n)$ of positive real numbers we consider the quantities
\begin{align}
T_{\ba,\bb}&=\sum_{k=1}^n\frac{a_k}{b_k},\label{E:T}\\
\noalign{\noindent\text{and}}
S_{\ba,\bb}^{(m)}&=\sum_{k=1}^n\sum_{l=1}^n\frac{a_ka_l}{(b_k+b_l)^m}
\quad \text{for $m=1,2,3$.}\label{E:S}
\end{align}

In Proposition \ref{pr1} we  prove that, for every 
positive integer $n$ and every vectors
$\ba=(a_1,\ldots,a_n)$ and $\bb=(b_1,\ldots,b_n)$ of positive real numbers, we have
\begin{equation}\label{E:Ineq}
\left(T_{\ba,\bb}\right)^2\leq 2S_{\ba,\bb}^{(2)}+2\sqrt{2}\sqrt{S_{\ba,\bb}^{(1)}
S_{\ba,\bb}^{(3)}}
\end{equation}

The difficulty does not reside in the proof
of \eqref{E:Ineq} but, in fact, it resides in showing that it is optimal in the sense that $2\sqrt{2}$ is the best possible constant. Precisely, we will prove in Theorem \ref{th2} that if for every 
positive integer $n$ and every vectors
$\ba=(a_1,\ldots,a_n)$ and $\bb=(b_1,\ldots,b_n)$ of positive real numbers, we have
$
\left(T_{\ba,\bb}\right)^2\leq 2S_{\ba,\bb}^{(2)}+\lambda\sqrt{S_{\ba,\bb}^{(1)}
S_{\ba,\bb}^{(3)}}
$
then $\lambda\geq 2\sqrt{2}$. 

This appears a difficult task, and requires tools from approximation theory and Fourier analysis. Indeed,
we will prove in Proposition \ref{pr2} that, for every $h>0$ there
exists two families of positive numbers $(a_j(h))_{j\in\ent}$ and
$(b_j(h))_{j\in\ent}$ such that
\begin{equation*}
\forall\,t\geq 0,\quad\abs{\frac{1}{(1+t)^2}-\sum_{j\in\ent}a_j(h)e^{-b_j(h)t}}\leq\frac{\delta(h)}{(1+t)^2}
\end{equation*}
with $\lim\limits_{h\to 0^+}\delta(h)=0$,
and this will be exploited in proving the announced optimality result.

\bg

\section{\sc The Main Results}\label{sec2}
In the next proposition, we give a proof of \eqref{E:Ineq}. 
\bg

\begin{proposition}\label{pr1}
For every 
positive integer $n$ and every vectors
$\ba=(a_1,\ldots,a_n)$ and $\bb=(b_1,\ldots,b_n)$ of positive real numbers, we have
\begin{equation*}
\left(T_{\ba,\bb}\right)^2\leq 2S_{\ba,\bb}^{(2)}+2\sqrt{2}\sqrt{S_{\ba,\bb}^{(1)}
S_{\ba,\bb}^{(3)}}.
\end{equation*}
\end{proposition}
\begin{proof}
Consider the function $f:\reel^+\vers\reel^+$, defined by \[f(t)=\sum_{j=1}^na_je^{-b_j t}\]
 Using the Cauchy-Schwarz inequality, we have
 \begin{align}\label{E:pr1}
 \left(\int_0^\infty f(t)dt\right)^2&=\left(\int_0^\infty \frac{1}{1+t}(1+t)f(t)dt\right)^2\notag\\
 &\leq\left(\int_0^\infty\frac{dt}{(1+t)^2}\right)\left(\int_0^\infty f^2(t)dt+2
  \int_0^\infty tf^2(t)dt+ \int_0^\infty t^2f^2(t)dt\right)\notag\\
 &=2\int_0^\infty tf^2(t)dt+\int_0^\infty f^2(t)dt+ \int_0^\infty t^2f^2(t)dt.
 \end{align}

Noting that $\int_0^\infty t^{m}e^{-bt}dt=\frac{m!}{b^{m+1}}$ for $m=0,1,2$, we obtain
\[
\int_0^\infty f(t)\,dt=T_{\ba,\bb},\qquad
\int_0^\infty t^{m}f^2(t)\,dt=m!\,S_{\ba,\bb}^{(m)},
\]
and the \eqref{E:pr1} becomes
\begin{equation}\label{E:pr2}
\left(T_{\ba,\bb}\right)^2\leq S_{\ba,\bb}^{(1)}+2S_{\ba,\bb}^{(2)}
+ 2S_{\ba,\bb}^{(3)}.
\end{equation}

Applying \eqref{E:pr2} to $\lambda\ba=(\lambda a_1,\ldots,\lambda a_n)$ and $\lambda\bb=(\lambda b_1,\ldots,\lambda b_n)$ for some
$\lambda>0$, we obtain
\begin{equation*}\label{E:pr3}
\left(T_{\ba,\bb}\right)^2\leq \lambda S_{\ba,\bb}^{(1)}+2S_{\ba,\bb}^{(2)}
+\frac{2}{\lambda}S_{\ba,\bb}^{(3)}
\end{equation*}
and the desired inequality follows  by 
choosing $\lambda=\sqrt{2S_{\ba,\bb}^{(3)}/S_{\ba,\bb}^{(1)}}$.
\end{proof}

\bg
Analyzing the preceding proof, we see that in order to prove the optimality of \eqref{E:Ineq}, and to have equality we need
the function $t\mapsto f(t)$ to be proportional  to
$t\mapsto 1/(1+t)^2$, but this is impossible since the first has an exponential decay at $+\infty$. This remark holds the idea of what we will do next!. We will look for ``almost'' equality by approximating $t\mapsto 1/(1+t)^2$ by a linear combination of 
decreasing exponentials with positive coefficients. The next Proposition \ref{pr2} provides us with the desired conclusion. But before we proceed, we will need the next two technical lemmas.
\bg
\begin{lemma}\label{lm0}
The necessary and sufficient condition, on the positive parameter $\lambda$, for the following inequality to hold, for $x\in\reel$, 
\begin{equation*}
\frac{\pi x(1+x^2)}{\sinh(\pi x)}\leq \frac{1}{\cosh^2(\lambda x)}
\end{equation*}
is that $\lambda\leq\lambda_0\egdef\sqrt{\frac{\pi^2}{6}-1}
\approx 0.803078$.
\end{lemma}
\begin{proof}
Suppose that the proposed inequality is satisfied for some $\lambda>0$
then we must have
\[\frac{(1+x^2)\cosh^2(x)-1}{x^2}\leq \frac{1}{x^2}\left(\frac{\sinh(\pi x)}{\pi x}-1\right)\]
for every nonzero $x$. Letting $x$ tend to $0$ we obtain $1+\lambda^2\leq \pi^2/6$.\bg
Conversely, let $\lambda_0=\sqrt{\frac{\pi^2}{6}-1}$, and consider the function
\[
f(x)=\frac{\sinh(\pi x)}{\pi x}-(1+x^2)\cosh^2(\lambda_0x)=
\frac{\sinh(\pi x)}{\pi x}-\frac{1}{2}(1+x^2)(1+\cosh(2\lambda_0x)).
\] 

The power series expansion of $f$ is given by
\[
f(x)=\sum_{n=2}^\infty (1-a_n)\frac{(\pi x)^{2n}}{(2n+1)!}
\] 
where
\[
a_n=(2n+1)\left(\frac{2\lambda_0}{\pi}\right)^{2n-2}
\left(\frac{1}{3}+\frac{ 2n^2-n-2}{\pi^2} \right)
\]
Now,
\[
\frac{a_{n+1}}{a_{n}}=
\left(\frac{2}{3}-\frac{4}{\pi^2}\right)\left(1+
\frac{2}{2n+1}\right)
\left(1+\frac{12n+3}{6n^2-3n+\pi^2-6}
\right)
\]
From this, it is straightforward to see that the sequence $\left(\frac{a_{n+1}}{a_{n}}\right)_{n\geq 2}$ is decreasing, and that $\frac{a_{3}}{a_{2}}\approx 0.8177<1$. Thus,
$a_n\leq a_2\approx 0.96531<1$ for every $n\geq 2$. This proves that $f(x)\geq 0$ for every real number $x$, and the proposed inequality follows for $\lambda\in[0,\lambda_0]$.
\end{proof}

\begin{lemma}\label{lm1}
For $t\geq 0$, let $f_t:\reel\to\reel$ be the function defined by
\begin{equation*}
f_t(x)= e^{2x-(1+t)e^{x}}.
\end{equation*}
Then the Fourier transform $\widehat{f_t}=\int_{\reel}f_t(x)e^{ix(\cdot)}\,dx $ of $f_t$ satisfies
\begin{equation*}
\abs{\widehat{f_t}(w)}=\frac{1}{(1+t)^2}\,\sqrt{\frac{\pi w(1+w^2)}{\sinh(\pi w)}}
\leq \frac{1}{(1+t)^2}\cdot \frac{1}{\cosh(\lambda_0 w)}.
\end{equation*}
\end{lemma}
\begin{proof}
Indeed we have
\begin{align*}
\widehat{f_t}(w)&=\int_{-\infty}^\infty f_t(x)e^{-iwx}\,dx\\
&=\int_{-\infty}^\infty e^{(2-iw)x} e^{-(1+t)e^{x}} \,dx,\qquad\text{setting $s\leftarrow (1+t)e^{x}$,}\\
&=\frac{1}{(1+t)^{2-iw}}\int_{0}^\infty s^{1-iw} e^{-s} \,ds
=\frac{\Gamma(2-iw)}{(1+t)^{2-iw}},
\end{align*}
where $\Gamma$ is the well-known Eulerian Gamma function \cite{wei}.
Therefore
\begin{align*}
\abs{\widehat{f_t}(w)}^2&=\frac{\Gamma(2-iw)}{(1+t)^{2-iw}}\cdot\frac{\Gamma(2+iw)}{(1+t)^{2+iw}}\\
&=\frac{1}{(1+t)^4}(1-iw)(1+iw)iw\Gamma(1-iw)\Gamma(iw)\\
&=\frac{1}{(1+t)^4}\cdot(1+w^2)\cdot\frac{i\pi w}{\sin(i\pi w)}\\
&=\frac{1}{(1+t)^4}\cdot(1+w^2)\cdot\frac{\pi w}{\sinh(\pi w)}.
\end{align*}
Here we used Euler's reflection formula for the
Gamma function: $\Gamma(z)\Gamma(1-z)=\frac{\pi z}{\sin(\pi z)}$, (see \cite[Chapter 6, formula 6.1.17]{abr}). Finally
\begin{equation}
\abs{\widehat{f_t}(w)}=\frac{1}{(1+t)^2}\,\sqrt{\frac{\pi w(1+w^2)}{\sinh(\pi w)}},
\end{equation}
and the proposed inequality follows from Lemma \ref{lm0}.
\end{proof}

\bg
In the next proposition we prove the announced approximation result. In fact, the approach consists of approximating the function
$t\mapsto \frac{1}{(1+t)^2}$, written as an integral, with
of a positive function of exponential type, using the
trapezoidal quadrature rule, making use of
Poisson's formula to yield a good control on the committed error. For more details on this approach, we refer the reader to
 \cite{Wal} and the references therein.  The details of the proof are provided for the convenience of the reader.
\bg

\begin{proposition}\label{pr2}
For  $h>0$ and $n\in\ent$, let
\begin{equation*}
a_n(h)=h \exp\left(2nh-e^{nh}\right),\qquad b_n(h)=e^{nh},
\end{equation*}
then
\begin{equation*}
\forall\,t\geq 0,\quad\abs{\frac{1}{(1+t)^2}-\sum_{n\in\ent}a_n(h)e^{-b_n(h)t}}\leq\frac{\delta(h)}{(1+t)^2}
\end{equation*}
with
\begin{equation*}
\delta(h)=\frac{4}{\exp\left(\frac{2\pi\lambda_0}{h}\right)-1},
\end{equation*}
where $\lambda_0$ was defined in Lemma \ref{lm0}.
\end{proposition}
\begin{proof}
Noting that, for $t\geq 0$ we have 
\begin{equation}\label{E:pr21}
\frac{1}{(1+t)^2}=\int_0^\infty u e^{-(1+t)u}\,du=\int_{-\infty}^\infty
f_t(x)\,dx
\end{equation}
where $f_t$ is the positive function defined in Lemma \ref{lm0}.
The function $f_t$ is super-exponentially decreasing for positive $x$ and exponentially decreasing for negative $x$. A simple upper bound for $f_t$ is obtained as follows,  for $x\geq 0$ we have
\begin{equation}\label{E:pr23}
2x-(1+t)e^x\leq 2x-e^x\leq 2e^{x-1}-e^x=(2-e)e^{x-1}\leq (2-e) x
\end{equation}
since $x\leq e^{x-1}$ for every real $x$. And, for $x<0$, we have
\begin{equation}\label{E:pr24}
2x-(1+t)e^x< 2x<(e-2) x
\end{equation}
Combining \eqref{E:pr23} and \eqref{E:pr24} we see that $f_t(x)\leq e^{(2-e)\abs{x}}$, for $x\in \reel$.  

This simple upper bound shows that the series
$
\sum_{n\in\ent}f_t(\cdot+n h)
$
is uniformly convergent on every compact subset of $\reel$. Therefore,
we define an $h$-periodic \textit{continuous} function $F_t$ by the formula
\begin{equation}\label{E:pr25}
F_t(x)=\sum_{n\in\ent}f_t(x+n h).
\end{equation}
Moreover, the exponential Fourier coefficients $(C_m(F_t))_{m\in\ent}$ of $F_t$
are given by
\begin{align}
C_m(F_t)&=\frac{1}{h}\int_0^hF_t(x)e^{-2i\pi m x/h}dx\notag\\
&=\frac{1}{h}\sum_{n\in\ent}\int_0^hf_t(x+n h)e^{-2i\pi m x/h}dx\notag\\
&=\frac{1}{h}\sum_{n\in\ent}\int_{nh}^{(n+1)h}f_t(x)e^{-2i\pi m x/h}dx\notag\\
&=\frac{1}{h}\int_{-\infty}^{\infty}f_t(x)e^{-2i\pi m x/h}dx=\frac{1}{h} \widehat{f_t}\left(\frac{2\pi m}{h}\right)
\end{align}
where $\widehat{f_t}$ is the Fourier transform of $f_t$. In particular, according to Lemma \ref{lm1},
the Fourier series of $F_t$ is normally convergent, and consequently it is equal to $F_t$. Taking the value at $x=0$ we get
\begin{equation}
h\sum_{n\in\ent}f_t(n h)=\sum_{m\in\ent}\widehat{f_t}
\left(\frac{2\pi m}{h}\right)
\end{equation}
Using \eqref{E:pr21} and Lemma \ref{lm1} we get
\begin{align*}
\abs{\frac{1}{(1+t)^2}-h\sum_{n\in\ent}f_t(n h)}&\leq 2\sum_{m=1}^\infty\abs{\widehat{f_t}
\left(\frac{2\pi m}{h}\right)}\\
&\leq \frac{2}{(1+t)^2} \sum_{m=1}^\infty\frac{1}{\cosh(2\pi\lambda_0 m/h)}\\
&\leq \frac{4}{(1+t)^2}\sum_{m=1}^\infty\exp\left(-\frac{2\pi\lambda_0m}{h}\right)=\frac{\delta(h)}{(1+t)^2},
\end{align*}
and the proposition follows.
\end{proof}
\bg
Now, we have what we need to prove the next result.
\bg
\begin{theorem}\label{th2}
Consider a positive real constant $\lambda$ such that, for every 
positive integer $n$ and every vectors
$\ba=(a_1,\ldots,a_n)$ and $\bb=(b_1,\ldots,b_n)$ of positive real numbers, we have
\begin{equation}\label{E:th20}
\left(T_{\ba,\bb}\right)^2\leq 2S_{\ba,\bb}^{(2)}+\lambda \sqrt{S_{\ba,\bb}^{(1)}
S_{\ba,\bb}^{(3)}}
\end{equation}
Then $\lambda\geq 2 \sqrt{2}$.
\end{theorem}
\begin{proof}
Consider $h>0$ and let the families $(a_n(h))_{n\in\ent}$ and $(b_n(h))_{n\in\ent}$ be defined as in Proposition \ref{pr2}. Accordingly we have
\[
\frac{1-\delta(h)}{(1+t)^2}\leq \sum_{n\in\ent}a_n(h)e^{-b_n(h)t}\leq \frac{1+\delta(h)}{(1+t)^2}
\]
we conclude that
\begin{equation}\label{E:th21}
1-\delta(h)\leq \int_0^\infty\left(\sum_{n\in\ent}a_n(h)e^{-b_n(h)t}\right)dt
=\sum_{n\in\ent}\frac{a_n(h)}{b_n(h)}
\end{equation}
and, for $m=0,1,2$,
\begin{equation}
\int_0^\infty t^m\left(\sum_{n\in\ent}a_n(h)e^{-b_n(h)t}\right)^2 dt
\leq(1+\delta(h))^2\int_0^\infty\frac{t^m}{(1+t)^4}dt
\end{equation}
This yields
\begin{align}
\sum_{(k,l)\in\ent^2}\frac{a_k(h)a_l(h)}{b_k(h)+b_l(h) } &\leq \frac{(1+\delta(h))^2}{3}\label{E:th230}\\
\sum_{(k,l)\in\ent^2}\frac{a_k(h)a_l(h)}{
(b_k(h)+b_l(h))^2} &\leq \frac{(1+\delta(h))^2}{6}\label{E:th231}\\
\sum_{(k,l)\in\ent^2}\frac{a_k(h)a_l(h)}{
(b_k(h)+b_l(h))^3} &\leq \frac{(1+\delta(h))^2}{6}\label{E:th232}
\end{align}
Now, according to \eqref{E:th21} there is a positive integer $\nu$ such that
\begin{equation}\label{E:th24}
1-2\delta(h)\leq \sum_{n=-\nu}^\nu\frac{a_n(h)}{b_n(h)}
\end{equation}
Taking $n=2\nu+1$,
$
\ba=\left(a_n(h)\right)_{-\nu\leq n\leq \nu},$ and
$\bb=\left(b_n(h)\right)_{-\nu\leq n\leq \nu}$,
we obtain using \eqref{E:th230}--\eqref{E:th24}:
 \begin{equation*}
 1-2\delta(h)\leq T_{\ba,\bb},
 \end{equation*}

\begin{equation*}
S_{\ba,\bb}^{(1)}\leq \frac{(1+\delta(h))^2}{3},\quad
S_{\ba,\bb}^{(2)}\leq \frac{(1+\delta(h))^2}{6},\quad
S_{\ba,\bb}^{(3)}\leq \frac{(1+\delta(h))^2}{6}
\end{equation*}
and from \eqref{E:th20} we conclude that
\[
(1-2\delta(h))^2\leq
 \frac{(1+\delta(h))^2}{3}\left(1+\frac{\lambda}{\sqrt{2}}\right).
 \]
 Letting $h$ tend to $0$ and recalling that
 $\lim_{h\to 0}\delta(h)=0$ we obtain $\lambda\geq 2\sqrt{2}$.
\end{proof}

\end{document}